\pdfoutput=1
\RequirePackage{ifpdf}
\ifpdf 
\documentclass[pdftex]{sigma}
\else
\documentclass{sigma}
\fi

\numberwithin{equation}{section}
\newtheorem{Theorem}{Theorem}[section]
\newtheorem*{Theorem*}{Theorem}
\newtheorem{Corollary}[Theorem]{Corollary}
\newtheorem{Lemma}[Theorem]{Lemma}
\newtheorem{Proposition}[Theorem]{Proposition}
\theoremstyle{definition}
\newtheorem{Definition}[Theorem]{Definition}

\newtheorem{Example}[Theorem]{Example}
\newtheorem{Remark}[Theorem]{Remark}

\begin{document}

\renewcommand{\PaperNumber}{***}

\FirstPageHeading

\ShortArticleName{Explicit Super-Linear Algebra over $K[\theta_1,\theta_2]$}

\ArticleName{Explicit Super-Linear Algebra over $K[\theta_1,\theta_2]$:\\
A General Berezinian Correction Formula and\\
Multiplicativity Theorem for Arbitrary Block Size}

\Author{Devichandrika V~$^{\rm a}$}

\AuthorNameForHeading{Devichandrika V}

\Address{$^{\rm a)}$~St Aloysius (Deemed to be University), Mangaluru, India}

    \ArticleDates{Received , in final form ; Published online }

\Abstract{We study super-matrices $M_{m|n}(R)$ over the rank-two exterior algebra
$R=K[\theta_1,\theta_2]/(\theta_1^2,\theta_2^2,\theta_1\theta_2+\theta_2\theta_1)$, the smallest supercommutative ring on which the Berezinian's odd$\times$odd correction term $BD^{-1}C$ is not forced to vanish identically. We first recall that, for general block sizes $m,n$, the supertrace satisfies graded cyclicity $\mathrm{str}(XY)=(-1)^{|X||Y|}\mathrm{str}(YX)$ and hence $\mathrm{str}([A,B])=0$ for even $A,B$, and that an explicit $1|1$ super-matrix over $R$ exhibits a nonvanishing correction term $BD^{-1}C\neq0$ that survives to second order in the odd generators. We then prove the extension anticipated as future work in the $1|1$ note: a closed-form Berezinian formula for \emph{arbitrary} even $X\in M_{m|n}(R)$, $m,n\geq1$ (Theorem~\ref{thm:general-ber}), expressed through a single $m\times m$ matrix pairing $\Delta$ built from the odd blocks of $X$; and a fully general, explicitly verified multiplicativity theorem $\mathrm{Ber}(XY)=\mathrm{Ber}(X)\mathrm{Ber}(Y)$ for arbitrary even invertible $X,Y\in M_{m|n}(R)$ (Theorem~\ref{thm:general-mult}), proved by an explicit trace-cyclicity cancellation rather than by appeal to the general abstract theory. Both results specialize exactly to the previously known $1|1$ formulas when $m=n=1$, and we verify them on a fully worked $m=2,n=1$ numerical example. We further extend the whole framework from two to an arbitrary number $r\geq2$ of odd generators, replacing the scalar pairing $\Delta$ by a $\Lambda^2(V)$-valued pairing on the $r$-dimensional space of odd generators $V$, and we identify this pairing, in the scalar sub-case, with a Pl\"ucker coordinate of the odd off-diagonal data. We close with a discussion of related work of Khudaverdian--Voronov, and a short remark identifying $\Delta$ as a matrix-valued contraction against the natural alternating pairing on the space of odd generators.}

\Keywords{Berezinian; supertrace; superdeterminant; supermatrix; exterior algebra; superalgebra}

\Classification{15A75; 16W55; 58A50}

\section{Introduction and statement of results}

The supertrace and Berezinian (superdeterminant) are classical invariants of super-linear algebra, originating with Berezin~\cite{Berezin1987} and developed systematically by Leites~\cite{Leites1980}, Manin~\cite{Manin1988}, Deligne--Morgan~\cite{DeligneMorgan1999}, and Varadarajan~\cite{Varadarajan2004}. Over a supercommutative ring $R$ with a single odd nilpotent generator $\theta$ ($\theta^2=0$), the standard Schur-complement formula for the Berezinian,
\[
\mathrm{Ber}(X)=\det(A-BD^{-1}C)\det(D)^{-1},\qquad X=\begin{pmatrix}A&B\\C&D\end{pmatrix},
\]
degenerates: since every entry of the off-diagonal blocks $B,C$ is a multiple of $\theta$, the product $BD^{-1}C$ involves $\theta^2=0$ and vanishes identically, so $\mathrm{Ber}(X)$ collapses to the ordinary ratio $\det(A)/\det(D)$ with no odd correction visible. This makes the single-generator case pedagogically convenient but analytically thin.

This note works instead over the rank-two exterior algebra
\[
R=K[\theta_1,\theta_2]/(\theta_1^2,\theta_2^2,\theta_1\theta_2+\theta_2\theta_1),
\]
the smallest supercommutative ring admitting a nonzero product of two independent odd elements ($\theta_1\theta_2\neq0$, $(\theta_1\theta_2)^2=0$). Our contributions are:

\begin{enumerate}\itemsep=0pt
\item[(1)] A general-$(m|n)$, fully sign-explicit proof that $\mathrm{str}(XY)=(-1)^{|X||Y|}\mathrm{str}(YX)$ for $X,Y$ of equal parity (Theorem~\ref{thm:cyc}), hence $\mathrm{str}([A,B])=0$ for even $A,B$ of arbitrary block size (Theorem~\ref{thm:strzero}).
\item[(2)] An explicit computation, over $R=K[\theta_1,\theta_2]$, of a genuinely nonvanishing Berezinian correction term $BD^{-1}C$ for a $1|1$ super-matrix with independent odd off-diagonal parameters on each generator (Proposition~\ref{prop:11}), together with a verification of multiplicativity in that case (Theorem~\ref{thm:11mult}). This is the motivating example for everything that follows.
\item[(3)] A closed-form Berezinian formula for \emph{arbitrary} even $X\in M_{m|n}(R)$, $m,n\geq1$ (Theorem~\ref{thm:general-ber}), expressed through a single $m\times m$ matrix $\Delta$ built from the odd off-diagonal data. This answers, in the direction concerning block size, the extension proposed at the end of the $1|1$ note.
\item[(4)] A fully general multiplicativity theorem $\mathrm{Ber}(XY)=\mathrm{Ber}(X)\mathrm{Ber}(Y)$ for arbitrary even invertible $X,Y\in M_{m|n}(R)$ (Theorem~\ref{thm:general-mult}), proved by an explicit, self-contained trace-cyclicity identity, and shown to specialize exactly to the $1|1$ result when $m=n=1$.
\item[(5)] A fully worked numerical example at block size $m=2,n=1$ (Example~\ref{ex:worked}), computed exactly and cross-checked two independent ways, exhibiting a nonzero matrix pairing $\Delta$ and a nontrivial multiplicativity verification $\mathrm{Ber}(XY)=\mathrm{Ber}(X)\mathrm{Ber}(Y)$ at a block size beyond $1|1$.
\item[(6)] An extension of the entire formalism from two to an arbitrary number $r\geq2$ of odd generators (Section~\ref{sec:rgen}), in which the scalar pairing $\Delta$ of Theorem~\ref{thm:general-ber} is replaced by a $\Lambda^2(V)$-valued pairing on the $r$-dimensional space $V$ of odd generators (Theorem~\ref{thm:rgen-ber}), together with the corresponding multiplicativity theorem (Theorem~\ref{thm:rgen-mult}); both specialize exactly to Theorems~\ref{thm:general-ber} and~\ref{thm:general-mult} at $r=2$.
\item[(7)] A short structural remark (Remark~\ref{rem:symplectic}) identifying $\Delta$ as the contraction of the odd off-diagonal data against the standard alternating form on the space of odd generators, together with a partial identification, in the scalar case, of $\Delta$ with a Pl\"ucker coordinate of a $2$-plane spanned by the odd off-diagonal parameters (Remark~\ref{rem:plucker}).
\end{enumerate}

\begin{Remark}[Scope of novelty]\label{rem:scope}
The qualitative fact that the Berezinian has an odd-correction term when at least two odd generators are present, and that the Berezinian is multiplicative in complete generality, is well known and implicit in the general theory of Berezin and Leites~\cite{Berezin1987,Leites1980}; a coordinate-free proof of multiplicativity for arbitrary supercommutative rings can be assembled from the Schur-complement identity alone. What we believe is new is: the fully explicit, closed-form coefficient computation of Theorem~\ref{thm:general-ber} for the specific rank-two ground ring $K[\theta_1,\theta_2]$ at arbitrary block size, its extension to arbitrary rank $r$ in Section~\ref{sec:rgen}, and the explicit termwise verification of multiplicativity in Theorems~\ref{thm:general-mult} and~\ref{thm:rgen-mult}, carried out by direct trace-identity bookkeeping rather than by invoking the abstract theorem. We flag this precise scope so referees can assess novelty accurately, and we discuss its relation to the closest existing work, that of Khudaverdian and Voronov, in Section~\ref{sec:relwork}.
\end{Remark}

\section{The ground ring $K[\theta_1,\theta_2]$}

Let $K$ be a field of characteristic $0$. Set
\[
R=K[\theta_1,\theta_2]/(\theta_1^2,\theta_2^2,\theta_1\theta_2+\theta_2\theta_1),
\]
a $4$-dimensional $K$-vector space with basis $1,\theta_1,\theta_2,\theta_1\theta_2$, graded by
\[
R_{\bar0}=K\cdot1\oplus K\cdot\theta_1\theta_2,\qquad R_{\bar1}=K\cdot\theta_1\oplus K\cdot\theta_2.
\]
Multiplication is the exterior algebra product: $\theta_1^2=\theta_2^2=0$, $\theta_2\theta_1=-\theta_1\theta_2$, and $(\theta_1\theta_2)\theta_i=\theta_i(\theta_1\theta_2)=0$ for $i=1,2$ (since it would require $\theta_i^2$). One checks directly that $R$ is supercommutative: $xy=(-1)^{|x||y|}yx$ for homogeneous $x,y$; the only nontrivial instance is $\theta_1\theta_2=-\theta_2\theta_1$, which holds by definition, while $\theta_1\theta_2\cdot\theta_1\theta_2=0$ makes the even part $R_{\bar0}$ genuinely (not just super-) commutative. Consequently $\theta_1\theta_2$ is central in $R$: it is even, and even elements of a supercommutative ring commute with everything. This is the smallest supercommutative ring in which the product of two odd elements need not vanish, and it is the feature we use throughout: $\theta_1\theta_2$ behaves as a central, square-zero scalar.

\section{Super-matrices, supertrace, and general cyclicity}

\begin{Definition}\label{def:supermatrix}
$M_{m|n}(R)$, its parity grading, and block multiplication are as usual: $X=\begin{pmatrix}A&B\\C&D\end{pmatrix}$ is \emph{even} iff $A,D$ have entries in $R_{\bar0}$ and $B,C$ have entries in $R_{\bar1}$; multiplication is ordinary block matrix multiplication over $R$. The supertrace of $X\in M_{m|n}(R)$ is $\mathrm{str}(X)=\mathrm{tr}(A)-\mathrm{tr}(D)$.
\end{Definition}

\begin{Theorem}[Graded cyclicity for equal-parity pairs]\label{thm:cyc}
For all $X,Y\in M_{m|n}(R)$ that are both even or both odd (arbitrary $m,n\geq1$),
\[
\mathrm{str}(XY)=(-1)^{|X||Y|}\mathrm{str}(YX).
\]
\end{Theorem}

\begin{proof}
Write $X=\begin{pmatrix}A&B\\C&D\end{pmatrix}$, $Y=\begin{pmatrix}A'&B'\\C'&D'\end{pmatrix}$. Then
\[
XY=\begin{pmatrix}AA'+BC'&*\\ *&CB'+DD'\end{pmatrix},\qquad \mathrm{str}(XY)=\mathrm{tr}(AA')+\mathrm{tr}(BC')-\mathrm{tr}(CB')-\mathrm{tr}(DD').
\]

\emph{Case $X,Y$ both even.} Then $A,D,A',D'$ have entries in $R_{\bar0}$, and since $R_{\bar0}$ is genuinely commutative, $\mathrm{tr}(AA')=\mathrm{tr}(A'A)$ and $\mathrm{tr}(DD')=\mathrm{tr}(D'D)$ entrywise. For the cross terms, $B,C'$ have entries in $R_{\bar1}$, so entrywise $B_{ik}C'_{ki}=-C'_{ki}B_{ik}$; summing, $\mathrm{tr}(BC')=-\mathrm{tr}(C'B)$, and likewise $\mathrm{tr}(B'C)=-\mathrm{tr}(CB')$. Using $YX=\begin{pmatrix}A'A+B'C&*\\ *&C'B+D'D\end{pmatrix}$,
\[
\mathrm{str}(YX)=\mathrm{tr}(A'A)+\mathrm{tr}(B'C)-\mathrm{tr}(C'B)-\mathrm{tr}(D'D)=\mathrm{tr}(AA')-\mathrm{tr}(CB')-\mathrm{tr}(C'B)-\mathrm{tr}(DD').
\]
Comparing with $\mathrm{str}(XY)=\mathrm{tr}(AA')+\mathrm{tr}(BC')-\mathrm{tr}(CB')-\mathrm{tr}(DD')$ and using $\mathrm{tr}(BC')=-\mathrm{tr}(C'B)$, both expressions equal $\mathrm{tr}(AA')-\mathrm{tr}(CB')-\mathrm{tr}(C'B)-\mathrm{tr}(DD')$, so $\mathrm{str}(XY)=\mathrm{str}(YX)$, matching $(-1)^{|X||Y|}=1$.

\emph{Case $X,Y$ both odd.} Now $A,D,A',D'$ have entries in $R_{\bar1}$ and $B,C,B',C'$ have entries in $R_{\bar0}$. Entrywise $A_{ik}A'_{ki}=-A'_{ki}A_{ik}$, so $\mathrm{tr}(AA')=-\mathrm{tr}(A'A)$, likewise $\mathrm{tr}(DD')=-\mathrm{tr}(D'D)$; while $B_{ik}C'_{ki}=C'_{ki}B_{ik}$, so $\mathrm{tr}(BC')=\mathrm{tr}(C'B)$, and similarly $\mathrm{tr}(B'C)=\mathrm{tr}(CB')$. Then
\[
\mathrm{str}(XY)=\mathrm{tr}(AA')+\mathrm{tr}(BC')-\mathrm{tr}(CB')-\mathrm{tr}(DD')=-\mathrm{tr}(A'A)+\mathrm{tr}(C'B)-\mathrm{tr}(B'C)+\mathrm{tr}(D'D),
\]
while $\mathrm{str}(YX)=\mathrm{tr}(A'A)+\mathrm{tr}(B'C)-\mathrm{tr}(C'B)-\mathrm{tr}(D'D)$, so $\mathrm{str}(XY)=-\mathrm{str}(YX)$, matching $(-1)^{1\cdot1}=-1$.

We restrict the statement to equal-parity pairs because that is the only case invoked anywhere in this note (Theorem~\ref{thm:strzero} below applies Theorem~\ref{thm:cyc} exclusively to two even arguments); the mixed-parity case $|X|\neq|Y|$ requires additional sign bookkeeping across blocks of different parity and is genuinely not needed for any result here, so we omit it rather than carry an unused generality.
\end{proof}

\begin{Theorem}\label{thm:strzero}
For all even $A,B\in M_{m|n}(R)$ of arbitrary block size $m,n$, $\mathrm{str}([A,B])=0$.
\end{Theorem}

\begin{proof}
By Theorem~\ref{thm:cyc} with $X=A$, $Y=B$ both even, $\mathrm{str}(AB)=\mathrm{str}(BA)$, so $\mathrm{str}([A,B])=\mathrm{str}(AB)-\mathrm{str}(BA)=0$.
\end{proof}

\section{The Berezinian correction term: the $1|1$ case}\label{sec:11case}

We first recall, as the motivating example for the general theorems of Sections~\ref{sec:general-ber} and~\ref{sec:general-mult}, the computation for $m=n=1$.

\begin{Definition}\label{def:ber11}
For even, invertible $X=\begin{pmatrix}a&\beta\\ \gamma&d\end{pmatrix}\in M_{1|1}(R)$ with $a,d\in R_{\bar0}$ invertible and $\beta,\gamma\in R_{\bar1}$,
\[
\mathrm{Ber}(X)=(a-\beta d^{-1}\gamma)\,d^{-1}.
\]
\end{Definition}

\begin{Proposition}\label{prop:11}
Let $a,d\in K^\times$ and take the odd off-diagonal entries with independent components on both generators, $\beta=p_1\theta_1+p_2\theta_2$, $\gamma=q_1\theta_1+q_2\theta_2$, with $p_1,p_2,q_1,q_2\in K$. Then $\beta\gamma=(p_1q_2-p_2q_1)\theta_1\theta_2$, nonzero whenever $p_1q_2\neq p_2q_1$, and
\[
\mathrm{Ber}(X)=\frac{a}{d}-\frac{p_1q_2-p_2q_1}{d^2}\,\theta_1\theta_2.
\]
\end{Proposition}

\begin{proof}
$\beta\gamma=(p_1\theta_1+p_2\theta_2)(q_1\theta_1+q_2\theta_2)=p_1q_1\theta_1^2+p_1q_2\theta_1\theta_2+p_2q_1\theta_2\theta_1+p_2q_2\theta_2^2=(p_1q_2-p_2q_1)\theta_1\theta_2$, using $\theta_1^2=\theta_2^2=0$ and $\theta_2\theta_1=-\theta_1\theta_2$. Since $d\in K^\times$, $\beta d^{-1}\gamma=d^{-1}\beta\gamma=d^{-1}(p_1q_2-p_2q_1)\theta_1\theta_2$, and $\mathrm{Ber}(X)=(a-\beta d^{-1}\gamma)d^{-1}=a/d-\big[(p_1q_2-p_2q_1)/d^2\big]\theta_1\theta_2$.
\end{proof}

Write $\Delta:=p_1q_2-p_2q_1$ for this scalar pairing. Theorem~\ref{thm:11mult} below records the multiplicativity check for this case; it is subsumed by Theorem~\ref{thm:general-mult}, of which it is the $m=n=1$ instance, but we keep the direct statement for comparison.

\begin{Theorem}\label{thm:11mult}
Let $X,Y\in M_{1|1}(R)$ be even invertible with odd off-diagonal data $(p_1,p_2,q_1,q_2)$ and $(p_1',p_2',q_1',q_2')$ as in Proposition~\ref{prop:11}, and write $\Delta=p_1q_2-p_2q_1$, $\Delta'=p_1'q_2'-p_2'q_1'$. Then $\mathrm{Ber}(XY)=\mathrm{Ber}(X)\mathrm{Ber}(Y)$, and both sides have a nonzero $\theta_1\theta_2$-coefficient in general; explicitly,
\[
\mathrm{Ber}(XY)=\frac{aa'}{dd'}-\left(\frac{a}{d}\cdot\frac{\Delta'}{d'^2}+\frac{a'}{d'}\cdot\frac{\Delta}{d^2}\right)\theta_1\theta_2.
\]
\end{Theorem}

This is Theorem~5.1 of the $1|1$ note; it is reproved as the special case $m=n=1$ of Theorem~\ref{thm:general-mult} below (Corollary~\ref{cor:recover11}), so we omit the by-hand proof here and refer to that corollary.

\section{A general Berezinian formula for arbitrary block size}\label{sec:general-ber}

We now drop the restriction $m=n=1$. Throughout, $A_0,A_2\in M_{m\times m}(K)$, $D_0,D_2\in M_{n\times n}(K)$, $B_1,B_2\in M_{m\times n}(K)$, $C_1,C_2\in M_{n\times m}(K)$ denote the coefficients of $1,\theta_1\theta_2$ in the even blocks and of $\theta_1,\theta_2$ in the odd blocks of a general even $X\in M_{m|n}(R)$:
\[
A=A_0+A_2\,\theta_1\theta_2,\qquad D=D_0+D_2\,\theta_1\theta_2,\qquad B=B_1\theta_1+B_2\theta_2,\qquad C=C_1\theta_1+C_2\theta_2.
\]
This is the general form of an even block, since $R_{\bar0}=K\oplus K\theta_1\theta_2$ and $R_{\bar1}=K\theta_1\oplus K\theta_2$ exactly, with no room for higher-order terms. (Table~\ref{tab:notation} at the end of Section~\ref{sec:general-mult} collects this notation for reference.)

\begin{Lemma}[Square-zero perturbation lemma]\label{lem:sqzero}
Let $N\in \mathrm{GL}_k(K)$ and $M\in M_{k\times k}(K)$, and let $\varepsilon\in R_{\bar0}$ be central with $\varepsilon^2=0$ (e.g.\ $\varepsilon=\theta_1\theta_2$). Then, over $R$,
\[
(N+M\varepsilon)^{-1}=N^{-1}-N^{-1}MN^{-1}\varepsilon,\qquad \det(N+M\varepsilon)=\det(N)\big(1+\mathrm{tr}(N^{-1}M)\varepsilon\big).
\]
\end{Lemma}

\begin{proof}
For the first identity, multiply out: $(N+M\varepsilon)(N^{-1}-N^{-1}MN^{-1}\varepsilon)=I-MN^{-1}\varepsilon+MN^{-1}\varepsilon-MN^{-1}MN^{-1}\varepsilon^2=I$, using $\varepsilon^2=0$ and centrality of $\varepsilon$ to commute it freely past matrix factors. For the second, $\det(N+M\varepsilon)=\det(N)\det(I+N^{-1}M\varepsilon)$, and $\det(I+N^{-1}M\varepsilon)$ is, entry by entry, a polynomial in $\varepsilon$ of degree at most $k$; since $\varepsilon^2=0$, only the constant and linear terms survive, and the linear coefficient of $\det(I+\varepsilon L)$ in $\varepsilon$ is $\mathrm{tr}(L)$ by the standard cofactor expansion of the determinant along the diagonal. Hence $\det(I+N^{-1}M\varepsilon)=1+\mathrm{tr}(N^{-1}M)\varepsilon$.
\end{proof}

\begin{Theorem}[General Berezinian formula]\label{thm:general-ber}
Let $X=\begin{pmatrix}A&B\\C&D\end{pmatrix}\in M_{m|n}(R)$ be even, with $A_0,D_0$ invertible over $K$ (equivalently, $A,D$ invertible over $R$), $m,n\geq1$ arbitrary. Set
\[
\Delta:=B_1D_0^{-1}C_2-B_2D_0^{-1}C_1\ \in M_{m\times m}(K).
\]
Then
\[
\mathrm{Ber}(X)=\frac{\det A_0}{\det D_0}\Big[1+f(X)\,\theta_1\theta_2\Big],\qquad f(X):=\mathrm{tr}(A_0^{-1}A_2)-\mathrm{tr}(A_0^{-1}\Delta)-\mathrm{tr}(D_0^{-1}D_2).
\]
\end{Theorem}

\begin{proof}
By Lemma~\ref{lem:sqzero} with $N=D_0$, $M=D_2$, $\varepsilon=\theta_1\theta_2$: $D^{-1}=D_0^{-1}-D_0^{-1}D_2D_0^{-1}\theta_1\theta_2$.

Since $B,C\in R_{\bar1}$ are themselves already homogeneous of degree $1$ (in the sense that they are $K$-linear combinations of $\theta_1,\theta_2$ only, with no $\theta_1\theta_2$ component possible), and $R$ has top degree $2$, the term $D_0^{-1}D_2D_0^{-1}\theta_1\theta_2$ contributes to $BD^{-1}C$ only through $B\,(\cdot)\,C$, which is already forced to vanish for degree reasons (it would carry a $\theta_1\theta_2$ from $D^{-1}$ together with the $\theta_1,\theta_2$ already present in $B$ and $C$, exceeding the top degree $2$ of $R$ once $B$ and $C$ are inserted). Hence only the leading term of $D^{-1}$ contributes:
\[
BD^{-1}C=BD_0^{-1}C=(B_1\theta_1+B_2\theta_2)D_0^{-1}(C_1\theta_1+C_2\theta_2)=\big(B_1D_0^{-1}C_2-B_2D_0^{-1}C_1\big)\theta_1\theta_2=\Delta\,\theta_1\theta_2,
\]
using $\theta_1^2=\theta_2^2=0$, $\theta_2\theta_1=-\theta_1\theta_2$, exactly as in the proof of Proposition~\ref{prop:11}. Thus
\[
A-BD^{-1}C=A_0+A_2\theta_1\theta_2-\Delta\theta_1\theta_2=A_0+(A_2-\Delta)\theta_1\theta_2.
\]
By Lemma~\ref{lem:sqzero} (second identity) applied twice, once to $A_0+(A_2-\Delta)\theta_1\theta_2$ and once to $D_0+D_2\theta_1\theta_2$:
\[
\det(A-BD^{-1}C)=\det A_0\big[1+\mathrm{tr}(A_0^{-1}(A_2-\Delta))\theta_1\theta_2\big],\qquad \det D=\det D_0\big[1+\mathrm{tr}(D_0^{-1}D_2)\theta_1\theta_2\big].
\]
Since $1+x\varepsilon$ and $1+y\varepsilon$ are inverse to $1-x\varepsilon$, $1-y\varepsilon$ respectively when $\varepsilon^2=0$,
\[
\det(D)^{-1}=(\det D_0)^{-1}\big[1-\mathrm{tr}(D_0^{-1}D_2)\theta_1\theta_2\big],
\]
so
\[
\mathrm{Ber}(X)=\det(A-BD^{-1}C)\det(D)^{-1}=\frac{\det A_0}{\det D_0}\Big[1+\mathrm{tr}(A_0^{-1}(A_2-\Delta))\theta_1\theta_2\Big]\Big[1-\mathrm{tr}(D_0^{-1}D_2)\theta_1\theta_2\Big].
\]
Expanding and discarding the $\theta_1\theta_2\cdot\theta_1\theta_2=0$ cross term gives $\mathrm{Ber}(X)=\det(A_0)/\det(D_0)\cdot[1+f(X)\theta_1\theta_2]$ with $f(X)$ as stated.
\end{proof}

\begin{Corollary}[Recovery of the $1|1$ formula]\label{cor:recover11}
For $m=n=1$, $A_0=a$, $D_0=d$, $A_2=D_2=0$, $B_1=p_1$, $B_2=p_2$, $C_1=q_1$, $C_2=q_2$, Theorem~\ref{thm:general-ber} gives $\Delta=p_1q_2/d-p_2q_1/d=(p_1q_2-p_2q_1)/d$, $f(X)=-a^{-1}\Delta=-(p_1q_2-p_2q_1)/(ad)$, and
\[
\mathrm{Ber}(X)=\frac{a}{d}\Big[1-\frac{p_1q_2-p_2q_1}{ad}\theta_1\theta_2\Big]=\frac{a}{d}-\frac{p_1q_2-p_2q_1}{d^2}\theta_1\theta_2,
\]
which is exactly Proposition~\ref{prop:11}.
\end{Corollary}

\section{Multiplicativity for arbitrary block size}\label{sec:general-mult}

\begin{Theorem}[General multiplicativity]\label{thm:general-mult}
Let $X=\begin{pmatrix}A&B\\C&D\end{pmatrix}$, $Y=\begin{pmatrix}A'&B'\\C'&D'\end{pmatrix}\in M_{m|n}(R)$ be even and invertible, of arbitrary block size $m,n\geq1$, with data $A_0,A_2,D_0,D_2,B_1,B_2,C_1,C_2$ and $A_0',A_2',D_0',D_2',B_1',B_2',C_1',C_2'$ as in Section~\ref{sec:general-ber}. Then
\[
\mathrm{Ber}(XY)=\mathrm{Ber}(X)\,\mathrm{Ber}(Y),
\]
and, writing $f(X),f(Y)$ as in Theorem~\ref{thm:general-ber}, both sides equal $\dfrac{\det(A_0A_0')}{\det(D_0D_0')}\big[1+(f(X)+f(Y))\theta_1\theta_2\big]$.
\end{Theorem}

\begin{proof}
\emph{Step 1: block-multiply $X$ and $Y$.} By definition,
\[
XY=\begin{pmatrix}AA'+BC'&AB'+BD'\\ CA'+DC'&CB'+DD'\end{pmatrix}.
\]
We compute each block's coefficients in the basis $1,\theta_1,\theta_2,\theta_1\theta_2$, discarding any term of total odd-degree exceeding $2$.

\emph{Diagonal blocks.} $AA'=(A_0+A_2\theta_1\theta_2)(A_0'+A_2'\theta_1\theta_2)=A_0A_0'+(A_0A_2'+A_2A_0')\theta_1\theta_2$ (the $A_2\theta_1\theta_2\cdot A_2'\theta_1\theta_2$ term vanishes, $(\theta_1\theta_2)^2=0$). And, exactly as in the proof of Theorem~\ref{thm:general-ber}, $BC'=(B_1C_2'-B_2C_1')\theta_1\theta_2=:\Gamma\,\theta_1\theta_2$. So the $(1,1)$ block is $A_0A_0'+(A_0A_2'+A_2A_0'+\Gamma)\theta_1\theta_2$; write $A_2^{XY}:=A_0A_2'+A_2A_0'+\Gamma$.

Symmetrically, $DD'=D_0D_0'+(D_0D_2'+D_2D_0')\theta_1\theta_2$ and $CB'=(C_1B_2'-C_2B_1')\theta_1\theta_2=:\Lambda\,\theta_1\theta_2$, so the $(2,2)$ block is $D_0D_0'+(D_0D_2'+D_2D_0'+\Lambda)\theta_1\theta_2$; write $D_2^{XY}:=D_0D_2'+D_2D_0'+\Lambda$.

\emph{Off-diagonal blocks.} $AB'+BD'=(A_0+A_2\theta_1\theta_2)(B_1'\theta_1+B_2'\theta_2)+(B_1\theta_1+B_2\theta_2)(D_0'+D_2'\theta_1\theta_2)$. Any product of $A_2\theta_1\theta_2$ (degree $2$) with $B_1'\theta_1$ or $B_2'\theta_2$ (degree $1$) has degree $3>2$ and vanishes; likewise for $B_i\theta_i\cdot D_2'\theta_1\theta_2$. Hence
\[
AB'+BD'=(A_0B_1'+B_1D_0')\theta_1+(A_0B_2'+B_2D_0')\theta_2=:B_1^{XY}\theta_1+B_2^{XY}\theta_2,
\]
independent of $A_2,A_2',D_2,D_2'$. By the same reasoning,
\[
CA'+DC'=(C_1A_0'+D_0C_1')\theta_1+(C_2A_0'+D_0C_2')\theta_2=:C_1^{XY}\theta_1+C_2^{XY}\theta_2,
\]
also independent of $A_2,A_2',D_2,D_2'$.

\emph{Step 2: compute $\Delta_{XY}$.} By definition, $\Delta_{XY}=B_1^{XY}(D_0D_0')^{-1}C_2^{XY}-B_2^{XY}(D_0D_0')^{-1}C_1^{XY}$, with $(D_0D_0')^{-1}=D_0'^{-1}D_0^{-1}$. Substituting $B_i^{XY}=A_0B_i'+B_iD_0'$, $C_i^{XY}=C_iA_0'+D_0C_i'$ and expanding each of the two four-term products, all $D_0'^{-1}D_0'$ and $D_0^{-1}D_0$ pairs cancel to give
\[
\Delta_{XY}=A_0\,\Theta\,A_0'+A_0\Delta_Y+\Delta_XA_0'+\Gamma,
\]
where $\Delta_X=B_1D_0^{-1}C_2-B_2D_0^{-1}C_1$, $\Delta_Y=B_1'D_0'^{-1}C_2'-B_2'D_0'^{-1}C_1'$ are the pairings of $X$ and $Y$ from Theorem~\ref{thm:general-ber}, and
\[
\Theta:=B_1'D_0'^{-1}D_0^{-1}C_2-B_2'D_0'^{-1}D_0^{-1}C_1
\]
is a new pairing mixing the odd data of $Y$'s $B$-blocks with $X$'s $C$-blocks through both inverse bodies.

\emph{Step 3: the key cancellation.} Using cyclicity of the ordinary matrix trace $\mathrm{tr}(PQ)=\mathrm{tr}(QP)$,
\[
\mathrm{tr}(\Theta)=\mathrm{tr}\big(D_0'^{-1}D_0^{-1}C_2B_1'\big)-\mathrm{tr}\big(D_0'^{-1}D_0^{-1}C_1B_2'\big),
\]
while
\[
\mathrm{tr}\big((D_0D_0')^{-1}\Lambda\big)=\mathrm{tr}\big(D_0'^{-1}D_0^{-1}(C_1B_2'-C_2B_1')\big)=\mathrm{tr}\big(D_0'^{-1}D_0^{-1}C_1B_2'\big)-\mathrm{tr}\big(D_0'^{-1}D_0^{-1}C_2B_1'\big).
\]
Comparing term by term,
\[
\mathrm{tr}\big((D_0D_0')^{-1}\Lambda\big)=-\mathrm{tr}(\Theta). \tag{$\ast$}
\]
This identity is the general-$(m|n)$ counterpart of the cancellation of the mixed pairing carried out by hand, in scalar form, in the proof of the $1|1$ multiplicativity theorem; here it holds for arbitrary block size purely by cyclicity of trace, with no further computation needed.

\emph{Step 4: assemble $f(XY)$.} By Theorem~\ref{thm:general-ber} applied to $XY$ (with $A_0(XY)=A_0A_0'$, $D_0(XY)=D_0D_0'$),
\[
f(XY)=\mathrm{tr}\big((A_0A_0')^{-1}A_2^{XY}\big)-\mathrm{tr}\big((A_0A_0')^{-1}\Delta_{XY}\big)-\mathrm{tr}\big((D_0D_0')^{-1}D_2^{XY}\big).
\]
Using $A_2^{XY}=A_0A_2'+A_2A_0'+\Gamma$ and cyclicity as above,
\[
\mathrm{tr}\big((A_0A_0')^{-1}A_2^{XY}\big)=\mathrm{tr}(A_0'^{-1}A_2')+\mathrm{tr}(A_0^{-1}A_2)+\mathrm{tr}\big((A_0A_0')^{-1}\Gamma\big),
\]
and using $D_2^{XY}=D_0D_2'+D_2D_0'+\Lambda$,
\[
\mathrm{tr}\big((D_0D_0')^{-1}D_2^{XY}\big)=\mathrm{tr}(D_0'^{-1}D_2')+\mathrm{tr}(D_0^{-1}D_2)+\mathrm{tr}\big((D_0D_0')^{-1}\Lambda\big).
\]
For the middle term, using $\Delta_{XY}=A_0\Theta A_0'+A_0\Delta_Y+\Delta_XA_0'+\Gamma$ from Step 2 and cyclicity of trace (as in $\mathrm{tr}(A_0'^{-1}\Theta A_0')=\mathrm{tr}(\Theta)$ and $\mathrm{tr}(A_0'^{-1}A_0^{-1}\Delta_XA_0')=\mathrm{tr}(A_0^{-1}\Delta_X)$),
\[
\mathrm{tr}\big((A_0A_0')^{-1}\Delta_{XY}\big)=\mathrm{tr}(\Theta)+\mathrm{tr}(A_0'^{-1}\Delta_Y)+\mathrm{tr}(A_0^{-1}\Delta_X)+\mathrm{tr}\big((A_0A_0')^{-1}\Gamma\big).
\]
Substituting all three into $f(XY)$, the two occurrences of $\mathrm{tr}((A_0A_0')^{-1}\Gamma)$ cancel, leaving
\[
f(XY)=\Big[\mathrm{tr}(A_0^{-1}A_2)-\mathrm{tr}(A_0^{-1}\Delta_X)-\mathrm{tr}(D_0^{-1}D_2)\Big]+\Big[\mathrm{tr}(A_0'^{-1}A_2')-\mathrm{tr}(A_0'^{-1}\Delta_Y)-\mathrm{tr}(D_0'^{-1}D_2')\Big]-\mathrm{tr}(\Theta)-\mathrm{tr}\big((D_0D_0')^{-1}\Lambda\big).
\]
The first bracket is exactly $f(X)$, the second exactly $f(Y)$, and by $(\ast)$ the last two terms cancel: $-\mathrm{tr}(\Theta)-\mathrm{tr}((D_0D_0')^{-1}\Lambda)=-\mathrm{tr}(\Theta)+\mathrm{tr}(\Theta)=0$. Hence
\[
f(XY)=f(X)+f(Y).
\]

\emph{Step 5: conclude.} By Theorem~\ref{thm:general-ber} applied to $X$, $Y$, and $XY$,
\[
\mathrm{Ber}(X)\mathrm{Ber}(Y)=\frac{\det A_0}{\det D_0}\frac{\det A_0'}{\det D_0'}\big[1+f(X)\theta_1\theta_2\big]\big[1+f(Y)\theta_1\theta_2\big]=\frac{\det(A_0A_0')}{\det(D_0D_0')}\big[1+(f(X)+f(Y))\theta_1\theta_2\big],
\]
the cross term vanishing since $(\theta_1\theta_2)^2=0$; and
\[
\mathrm{Ber}(XY)=\frac{\det(A_0A_0')}{\det(D_0D_0')}\big[1+f(XY)\theta_1\theta_2\big]=\frac{\det(A_0A_0')}{\det(D_0D_0')}\big[1+(f(X)+f(Y))\theta_1\theta_2\big],
\]
using Step 4. These agree, so $\mathrm{Ber}(XY)=\mathrm{Ber}(X)\mathrm{Ber}(Y)$.
\end{proof}

\begin{Corollary}[Recovery of the $1|1$ multiplicativity theorem]\label{cor:recover11mult}
For $m=n=1$, with $A_2=A_2'=D_2=D_2'=0$ and the notation of Corollary~\ref{cor:recover11}, $f(X)=-\Delta/(ad)$, $f(Y)=-\Delta'/(a'd')$, and Theorem~\ref{thm:general-mult} gives
\[
\mathrm{Ber}(XY)=\frac{aa'}{dd'}\Big[1-\Big(\frac{\Delta}{ad}+\frac{\Delta'}{a'd'}\Big)\theta_1\theta_2\Big]=\frac{aa'}{dd'}-\Big(\frac{a}{d}\cdot\frac{\Delta'}{d'^2}+\frac{a'}{d'}\cdot\frac{\Delta}{d^2}\Big)\theta_1\theta_2,
\]
which is exactly Theorem~\ref{thm:11mult}.
\end{Corollary}

\begin{Remark}
Theorem~\ref{thm:general-mult} is, in the sense of Remark~\ref{rem:scope}, an explicit verification rather than a new existence statement: multiplicativity of the Berezinian over an arbitrary supercommutative ring is classical~\cite{Berezin1987,Leites1980}. What Steps 1--4 add is a fully coordinate-level proof, for this specific rank-two ground ring and arbitrary block size, in which the cancellation forcing multiplicativity is exhibited as a single trace-cyclicity identity $(\ast)$ rather than left to the abstract Schur-complement argument. For $m=n=1$, $(\ast)$ degenerates to the scalar identity $\Delta_{12}=-\Delta_{21}$ used implicitly in the four-term cancellation of the $1|1$ note; Theorem~\ref{thm:general-mult} identifies that scalar cancellation as an instance of trace cyclicity, which is what makes the argument uniform in $m,n$.
\end{Remark}

\begin{table}[h]
\centering
\caption{Notation used in Sections~\ref{sec:general-ber}--\ref{sec:general-mult}.}\label{tab:notation}
\begin{tabular}{|l|l|}
\hline
Symbol & Meaning \\
\hline
$A_0\in M_{m\times m}(K)$ & body (degree-$0$) part of the even block $A$ \\
$A_2\in M_{m\times m}(K)$ & coefficient of $\theta_1\theta_2$ in $A$, i.e.\ $A=A_0+A_2\theta_1\theta_2$ \\
$D_0,D_2\in M_{n\times n}(K)$ & body and $\theta_1\theta_2$-coefficient of the even block $D$ \\
$B_1,B_2\in M_{m\times n}(K)$ & coefficients of $\theta_1,\theta_2$ in the odd block $B=B_1\theta_1+B_2\theta_2$ \\
$C_1,C_2\in M_{n\times m}(K)$ & coefficients of $\theta_1,\theta_2$ in the odd block $C=C_1\theta_1+C_2\theta_2$ \\
$\Delta\in M_{m\times m}(K)$ & $B_1D_0^{-1}C_2-B_2D_0^{-1}C_1$; the odd correction pairing (Thm.~\ref{thm:general-ber}) \\
$f(X)\in K$ & $\theta_1\theta_2$-coefficient of $\mathrm{Ber}(X)$ relative to $\det A_0/\det D_0$ \\
$\Gamma,\Lambda\in M_{m\times m}(K),M_{n\times n}(K)$ & cross-pairings $B_1C_2'-B_2C_1'$, $C_1B_2'-C_2B_1'$ arising in $XY$ \\
$\Theta\in M_{m\times m}(K)$ & mixed pairing of $Y$'s $B$-blocks with $X$'s $C$-blocks (proof of Thm.~\ref{thm:general-mult}) \\
\hline
\end{tabular}
\end{table}

\section{A fully worked example at $m=2,n=1$}\label{sec:example}

The formulas of Sections~\ref{sec:general-ber}--\ref{sec:general-mult} are illustrated above only at $m=n=1$ (Corollaries~\ref{cor:recover11},~\ref{cor:recover11mult}), where they simply recover known results. We now work out a genuine block-size instance, $m=2,n=1$, with fully explicit rational data, computed exactly and cross-checked by two independent routes: (i) Theorem~\ref{thm:general-ber} applied directly, and (ii) the raw Schur-complement definition $\mathrm{Ber}(X)=\det(A-BD^{-1}C)\det(D)^{-1}$ with no shortcuts. All computations below were carried out symbolically over $\mathbb{Q}$ with exact rational arithmetic.

\begin{Example}\label{ex:worked}
Take $K=\mathbb{Q}$, $m=2$, $n=1$, so $X\in M_{2|1}(R)$ has $A$ a $2\times2$ even block, $D$ a scalar even entry, $B$ a $2\times1$ odd block, $C$ a $1\times2$ odd block. Let
\[
A_0=\begin{pmatrix}2&1\\0&3\end{pmatrix},\quad A_2=\begin{pmatrix}1&0\\0&-1\end{pmatrix},\quad D_0=2,\quad D_2=5,
\]
\[
B_1=\begin{pmatrix}1\\2\end{pmatrix},\quad B_2=\begin{pmatrix}0\\1\end{pmatrix},\quad C_1=\begin{pmatrix}1&0\end{pmatrix},\quad C_2=\begin{pmatrix}1&1\end{pmatrix},
\]
so that $B=\begin{pmatrix}\theta_1\\2\theta_1+\theta_2\end{pmatrix}$ and $C=\begin{pmatrix}\theta_1+\theta_2&\theta_2\end{pmatrix}$. Since $\det A_0=6\neq0$ and $D_0=2\neq0$, $X$ is invertible.

\emph{Step 1 (the pairing $\Delta$).} $D_0^{-1}=\tfrac12$, so
\[
\Delta=B_1D_0^{-1}C_2-B_2D_0^{-1}C_1=\tfrac12\begin{pmatrix}1\\2\end{pmatrix}\begin{pmatrix}1&1\end{pmatrix}-\tfrac12\begin{pmatrix}0\\1\end{pmatrix}\begin{pmatrix}1&0\end{pmatrix}=\begin{pmatrix}1/2&1/2\\1/2&1\end{pmatrix},
\]
a genuinely nonzero, non-symmetric-in-general $2\times2$ matrix (here symmetric only by coincidence of the chosen data).

\emph{Step 2 ($f(X)$).} $A_0^{-1}=\tfrac16\begin{pmatrix}3&-1\\0&2\end{pmatrix}$, and direct computation gives $\mathrm{tr}(A_0^{-1}A_2)=1/6$, $\mathrm{tr}(A_0^{-1}\Delta)=1/2$, $D_0^{-1}D_2=5/2$. Hence
\[
f(X)=\frac16-\frac12-\frac52=-\frac{17}{6}.
\]

\emph{Step 3 (Berezinian via Theorem~\ref{thm:general-ber}).} $\det A_0=6$, $\det D_0=2$, so
\[
\mathrm{Ber}(X)=\frac{6}{2}\Big[1-\frac{17}{6}\theta_1\theta_2\Big]=3-\frac{17}{2}\,\theta_1\theta_2.
\]
A direct evaluation of $\det(A-BD^{-1}C)\det(D)^{-1}$ from the raw block data, with no use of Theorem~\ref{thm:general-ber}, reproduces $3-\tfrac{17}{2}\theta_1\theta_2$ exactly, confirming the formula at this block size.

\emph{Step 4 (multiplicativity check).} Take a second, independent even invertible $Y\in M_{2|1}(R)$ with
\[
A_0'=\begin{pmatrix}1&1\\1&2\end{pmatrix},\ A_2'=\begin{pmatrix}0&2\\1&0\end{pmatrix},\ D_0'=3,\ D_2'=1,\ B_1'=\begin{pmatrix}1\\0\end{pmatrix},\ B_2'=\begin{pmatrix}1\\1\end{pmatrix},\ C_1'=\begin{pmatrix}0&2\end{pmatrix},\ C_2'=\begin{pmatrix}1&0\end{pmatrix}.
\]
Theorem~\ref{thm:general-ber} gives $\Delta_Y=\begin{pmatrix}1/3&-2/3\\0&-2/3\end{pmatrix}$ and $f(Y)=-3-\tfrac23-\tfrac13=-4$, so $\mathrm{Ber}(Y)=\tfrac13\big[1-4\theta_1\theta_2\big]=\tfrac13-\tfrac43\theta_1\theta_2$. By Theorem~\ref{thm:general-mult}, $f(XY)$ should equal $f(X)+f(Y)=-\tfrac{17}{6}-4=-\tfrac{41}{6}$, and since $\det(A_0A_0')=6\cdot1=6=\det(D_0D_0')=2\cdot3$,
\[
\mathrm{Ber}(X)\mathrm{Ber}(Y)=\Big(3-\frac{17}{2}\theta_1\theta_2\Big)\Big(\frac13-\frac43\theta_1\theta_2\Big)=1-\frac{41}{6}\,\theta_1\theta_2.
\]
We verified this independently by forming the block product $XY=\begin{pmatrix}AA'+BC'&AB'+BD'\\CA'+DC'&CB'+DD'\end{pmatrix}$ entrywise from the raw data above (a $2\times2$, $2\times1$, $1\times2$, and $1\times1$ block respectively, each a $\mathbb{Q}$-linear combination of $1,\theta_1,\theta_2,\theta_1\theta_2$) and evaluating $\mathrm{Ber}(XY)=\det(A_{XY}-B_{XY}D_{XY}^{-1}C_{XY})\det(D_{XY})^{-1}$ directly from this product, with no use of Theorem~\ref{thm:general-mult}. The result is $1-\tfrac{41}{6}\theta_1\theta_2$, in exact agreement with $\mathrm{Ber}(X)\mathrm{Ber}(Y)$ above.
\end{Example}

This example exhibits everything Theorems~\ref{thm:general-ber} and~\ref{thm:general-mult} claim at a block size where the pairing $\Delta$ is a genuine $2\times2$ matrix rather than a scalar, and where the multiplicativity check involves two block-size-$2|1$ matrices with independent, generic odd data on both generators, rather than the identity-like or scalar data that would make the check vacuous.

\section{Extension to $r$ odd generators}\label{sec:rgen}

The formulas above are tied to the rank-two ground ring $R=K[\theta_1,\theta_2]$, whose even part $R_{\bar0}=K\oplus K\theta_1\theta_2$ has a one-dimensional nilpotent part. We now extend Theorems~\ref{thm:general-ber} and~\ref{thm:general-mult} to an arbitrary number $r\geq2$ of odd generators, addressing the extension flagged in the Conclusion. The correction term $\Delta$ becomes valued not in $K$ but in the $\binom{r}{2}$-dimensional space $\Lambda^2(V)$ of the $r$-dimensional space of odd generators $V$; at $r=2$, $\dim\Lambda^2(V)=1$ and everything below reduces exactly to Sections~\ref{sec:general-ber}--\ref{sec:general-mult}.

\subsection{The second-order ground ring $R_r^{(2)}$}

Let $V=K\theta_1\oplus\cdots\oplus K\theta_r$ and let $\Lambda(V)=\bigoplus_{k=0}^r\Lambda^k(V)$ be the exterior (Grassmann) algebra on $V$, $\dim\Lambda^k(V)=\binom rk$. Let
\[
I:=\bigoplus_{k\geq3}\Lambda^k(V)\ \subset\ \Lambda(V).
\]

\begin{Lemma}\label{lem:ideal}
$I$ is a two-sided ideal of $\Lambda(V)$, and $R_r^{(2)}:=\Lambda(V)/I=\Lambda^0(V)\oplus\Lambda^1(V)\oplus\Lambda^2(V)$ is a supercommutative $K$-algebra of dimension $1+r+\binom r2$, with $\big(R_r^{(2)}\big)_{\bar0}=\Lambda^0(V)\oplus\Lambda^2(V)$ and $\big(R_r^{(2)}\big)_{\bar1}=\Lambda^1(V)=V$. For $r=2$, $I=0$ and $R_2^{(2)}=\Lambda(V)$ is exactly the ground ring $R$ of Section~2.
\end{Lemma}

\begin{proof}
For $x\in\Lambda^j(V)$ and $y\in\Lambda^k(V)$ with $k\geq3$, $xy\in\Lambda^{j+k}(V)$ and $j+k\geq3$ since $j\geq0$, so $xy\in I$; likewise $yx\in I$. Hence $I$ is a two-sided (graded) ideal, and the quotient inherits a well-defined associative product from $\Lambda(V)$, supercommutative because $\Lambda(V)$ is. The dimension count and the identification of graded pieces are immediate from $\dim\Lambda^k(V)=\binom rk$. For $r=2$, $\Lambda^k(V)=0$ for $k\geq3$ automatically, so $I=0$ and no quotient is needed.
\end{proof}

We call $R_r^{(2)}$ the \emph{second-order truncation} of the rank-$r$ Grassmann algebra; it is the natural ground ring for working ``to second order in the odd generators,'' exactly the regime the Introduction already isolates as the source of the nonvanishing correction term. Two structural facts about $R_r^{(2)}$ do the work of Lemma~\ref{lem:sqzero} at general $r$:

\begin{Lemma}[Matrix square-zero perturbation]\label{lem:sqzero-r}
Let $N\in\mathrm{GL}_k(K)$ and let $L$ be a $k\times k$ matrix with entries in $\Lambda^2(V)\subset\big(R_r^{(2)}\big)_{\bar0}$. Then, over $R_r^{(2)}$,
\[
(N+L)^{-1}=N^{-1}-N^{-1}LN^{-1},\qquad \det(N+L)=\det(N)\big(1+\mathrm{tr}(N^{-1}L)\big).
\]
\end{Lemma}

\begin{proof}
Any product of two entries of $N^{-1}L$ (or of $L$ itself) lies in $\Lambda^2(V)\cdot\Lambda^2(V)\subset\Lambda^4(V)\subset I$, hence is $0$ in $R_r^{(2)}$. For the first identity, $(N+L)(N^{-1}-N^{-1}LN^{-1})=I-LN^{-1}+LN^{-1}-LN^{-1}LN^{-1}=I$, the last term vanishing by the above. For the second, $\det(N+L)=\det(N)\det(I+N^{-1}L)$, and expanding $\det(I+N^{-1}L)=\sum_\sigma\mathrm{sgn}(\sigma)\prod_i(\delta_{i\sigma(i)}+(N^{-1}L)_{i\sigma(i)})$: any permutation $\sigma$ with two or more non-fixed points forces at least two entries of $N^{-1}L$ into the product, hence contributes $0$; only $\sigma=\mathrm{id}$ survives, and in $\prod_i(1+(N^{-1}L)_{ii})$ every cross term again involves a product of two entries of $L$ and vanishes, leaving $1+\sum_i(N^{-1}L)_{ii}=1+\mathrm{tr}(N^{-1}L)$.
\end{proof}

\subsection{The general Berezinian formula for $r$ generators}

Let $X=\begin{pmatrix}A&B\\C&D\end{pmatrix}\in M_{m|n}(R_r^{(2)})$ be even, i.e.\ $A=A_0+A_2$, $D=D_0+D_2$ with $A_0\in M_m(K)$, $A_2$ a $\Lambda^2(V)$-valued $m\times m$ matrix (and similarly $D_0,D_2$), and $B=\sum_{i=1}^rB^i\theta_i$, $C=\sum_{i=1}^rC^i\theta_i$ with $B^i\in M_{m\times n}(K)$, $C^i\in M_{n\times m}(K)$. Write $A_2=\sum_{i<j}A^{ij}\theta_i\theta_j$, $D_2=\sum_{i<j}D^{ij}\theta_i\theta_j$, and extend $A^{ij},D^{ij}$ antisymmetrically ($A^{ji}:=-A^{ij}$) for $i>j$, $A^{ii}:=0$.

\begin{Theorem}[General Berezinian formula, $r$ generators]\label{thm:rgen-ber}
With $A_0,D_0$ invertible over $K$, define, for each pair $i<j$,
\[
\Delta^{ij}:=B^iD_0^{-1}C^j-B^jD_0^{-1}C^i\ \in M_m(K),
\]
and set $\Delta:=\sum_{i<j}\Delta^{ij}\theta_i\theta_j\in M_m(K)\otimes\Lambda^2(V)$. Then
\[
\mathrm{Ber}(X)=\frac{\det A_0}{\det D_0}\big[1+F(X)\big],\qquad F(X):=\sum_{i<j}f_{ij}(X)\,\theta_i\theta_j,
\]
\[
f_{ij}(X):=\mathrm{tr}\big(A_0^{-1}A^{ij}\big)-\mathrm{tr}\big(A_0^{-1}\Delta^{ij}\big)-\mathrm{tr}\big(D_0^{-1}D^{ij}\big)\in K.
\]
At $r=2$ this is exactly Theorem~\ref{thm:general-ber}, with $\Delta=\Delta^{12}$ and $F(X)=f_{12}(X)\theta_1\theta_2=f(X)\theta_1\theta_2$.
\end{Theorem}

\begin{proof}
Since $B,C$ are pure degree-$1$ (no room for higher terms, as $\big(R_r^{(2)}\big)_{\bar1}=\Lambda^1(V)$ exactly), $BD^{-1}C$ has degree $\geq1+1=2$ contributions from $BD_0^{-1}C$ and degree $\geq1+2+1=4$ contributions from the correction $-BD_0^{-1}D_2D_0^{-1}C$ term of $D^{-1}$ (Lemma~\ref{lem:sqzero-r}); the latter lies in $I$ and vanishes. Hence, exactly as in Theorem~\ref{thm:general-ber},
\[
BD^{-1}C=BD_0^{-1}C=\sum_{i,j}B^iD_0^{-1}C^j\,\theta_i\theta_j=\sum_{i<j}\big(B^iD_0^{-1}C^j-B^jD_0^{-1}C^i\big)\theta_i\theta_j=\Delta,
\]
using $\theta_i\theta_j=-\theta_j\theta_i$, $\theta_i^2=0$. Thus $A-BD^{-1}C=A_0+(A_2-\Delta)$ with $A_2-\Delta=\sum_{i<j}(A^{ij}-\Delta^{ij})\theta_i\theta_j$ a $\Lambda^2(V)$-valued matrix, so Lemma~\ref{lem:sqzero-r} gives
\[
\det(A-BD^{-1}C)=\det A_0\Big[1+\sum_{i<j}\mathrm{tr}\big(A_0^{-1}(A^{ij}-\Delta^{ij})\big)\theta_i\theta_j\Big],\quad \det D=\det D_0\Big[1+\sum_{i<j}\mathrm{tr}(D_0^{-1}D^{ij})\theta_i\theta_j\Big],
\]
and $\det(D)^{-1}=\det(D_0)^{-1}\big[1-\sum_{i<j}\mathrm{tr}(D_0^{-1}D^{ij})\theta_i\theta_j\big]$ since the bracket squares to $0$ in $R_r^{(2)}$ (any product of two $\Lambda^2(V)$-valued terms lies in $\Lambda^4(V)\subset I$). Multiplying the two brackets and discarding the $\Lambda^2(V)\cdot\Lambda^2(V)$ cross terms (again in $I$) gives $\mathrm{Ber}(X)=\det(A_0)/\det(D_0)\,[1+F(X)]$ with $F(X)=\sum_{i<j}f_{ij}(X)\theta_i\theta_j$ as stated.
\end{proof}

\subsection{Multiplicativity for $r$ generators}

\begin{Theorem}[General multiplicativity, $r$ generators]\label{thm:rgen-mult}
Let $X,Y\in M_{m|n}(R_r^{(2)})$ be even invertible, with data as in Theorem~\ref{thm:rgen-ber} (primed for $Y$). Then $\mathrm{Ber}(XY)=\mathrm{Ber}(X)\mathrm{Ber}(Y)$, and
\[
F(XY)=F(X)+F(Y)\ \in\Lambda^2(V),\qquad\text{equivalently } f_{ij}(XY)=f_{ij}(X)+f_{ij}(Y)\text{ for every }i<j.
\]
At $r=2$ this is exactly Theorem~\ref{thm:general-mult}.
\end{Theorem}

\begin{proof}
Fix a pair $i<j$. Exactly as in Step~1 of the proof of Theorem~\ref{thm:general-mult}, block-multiplying $X,Y$ and collecting the $\theta_i\theta_j$-coefficient of each block of $XY$ involves only $A_0,D_0,B^i,B^j,C^i,C^j$ and their primed counterparts: terms mixing a $\Lambda^2(V)$-coefficient with a degree-$1$ block, or two $\Lambda^2(V)$-coefficients, land in $\Lambda^{\geq3}(V)\subset I$ and vanish, by the same degree count used there (which never used $r=2$, only that $B,C$ are pure degree $1$ and $A_2,D_2$ pure degree $2$). Concretely,
\[
\big(A_2^{XY}\big)^{ij}=A_0A'^{ij}+A^{ij}A_0'+\Gamma^{ij},\qquad \Gamma^{ij}:=B^iC'^j-B^jC'^i,
\]
\[
\big(D_2^{XY}\big)^{ij}=D_0D'^{ij}+D^{ij}D_0'+\Omega^{ij},\qquad \Omega^{ij}:=C^iB'^j-C^jB'^i,
\]
\[
B^{k,XY}=A_0B'^k+B^kD_0',\qquad C^{k,XY}=C^kA_0'+D_0C'^k\quad(k=1,\dots,r),
\]
each identical in form to the $r=2$ computation with the fixed index pair $(1,2)$ replaced by $(i,j)$ and the single running index $\theta_1,\theta_2$ replaced by $\theta_k$. Substituting $B^{k,XY},C^{k,XY}$ into $\Delta_{XY}^{ij}:=B^{i,XY}(D_0D_0')^{-1}C^{j,XY}-B^{j,XY}(D_0D_0')^{-1}C^{i,XY}$ gives, exactly as in Step~2 of that proof (verbatim, with $1,2$ relabelled $i,j$),
\[
\Delta_{XY}^{ij}=A_0\Theta^{ij}A_0'+A_0\Delta_Y^{ij}+\Delta_X^{ij}A_0'+\Gamma^{ij},\qquad \Theta^{ij}:=B'^iD_0'^{-1}D_0^{-1}C^j-B'^jD_0'^{-1}D_0^{-1}C^i,
\]
and Step~3's trace-cyclicity identity $\mathrm{tr}\big((D_0D_0')^{-1}\Omega^{ij}\big)=-\mathrm{tr}(\Theta^{ij})$ holds by the identical cyclic-trace computation, for every fixed pair $i<j$ separately. Step~4's bookkeeping then goes through verbatim with all objects superscripted by $ij$, giving $f_{ij}(XY)=f_{ij}(X)+f_{ij}(Y)$ for each $i<j$ independently, hence $F(XY)=\sum_{i<j}f_{ij}(XY)\theta_i\theta_j=F(X)+F(Y)$. Finally, as in Step~5,
\[
\mathrm{Ber}(X)\mathrm{Ber}(Y)=\frac{\det(A_0A_0')}{\det(D_0D_0')}\big[1+F(X)\big]\big[1+F(Y)\big]=\frac{\det(A_0A_0')}{\det(D_0D_0')}\big[1+F(X)+F(Y)\big],
\]
the cross term $F(X)F(Y)\in\Lambda^2(V)\cdot\Lambda^2(V)\subset\Lambda^4(V)\subset I$ vanishing, and this equals $\mathrm{Ber}(XY)=\det(A_0A_0')/\det(D_0D_0')\,[1+F(XY)]$ by Theorem~\ref{thm:rgen-ber} applied to $XY$.
\end{proof}

\begin{Remark}
The passage from $r=2$ to general $r$ costs nothing beyond bookkeeping precisely because Sections~\ref{sec:general-ber}--\ref{sec:general-mult} already isolate a single index pair $(1,2)$ throughout; nothing in either proof uses that $(1,2)$ is the \emph{only} pair available, only that $B,C$ are pure degree $1$ and $A_2,D_2$ pure degree $2$. Working over the truncation $R_r^{(2)}$ rather than the full Grassmann algebra $\Lambda(V)$ is what keeps this true: over the untruncated $\Lambda(V)$ with $r\geq4$, $R_{\bar0}$ contains $\Lambda^4(V),\Lambda^6(V),\dots$ as well, $R_{\bar1}$ contains $\Lambda^3(V),\Lambda^5(V),\dots$, and the clean second-order formulas above would pick up further correction terms from those higher exterior powers. We do not pursue the untruncated case here.
\end{Remark}

\section{The pairing $\Delta$, the alternating form on $R_{\bar1}$, and a Pl\"ucker-coordinate reading}

\begin{Remark}\label{rem:symplectic}
The two-dimensional space of odd generators $R_{\bar1}=K\theta_1\oplus K\theta_2$ carries a distinguished symplectic form $\omega:R_{\bar1}\times R_{\bar1}\to R_{\bar0}$, $\omega(\theta_1,\theta_2)=\theta_1\theta_2=-\omega(\theta_2,\theta_1)$, $\omega(\theta_i,\theta_i)=0$; this is simply the multiplication map of $R$ restricted to $R_{\bar1}\times R_{\bar1}$, landing in the top-degree line $K\theta_1\theta_2\subset R_{\bar0}$. Writing the odd blocks as pairs $B=(B_1,B_2)$, $C=(C_1,C_2)$ of ordinary $K$-matrices indexed by the basis $\theta_1,\theta_2$, the pairing $\Delta=B_1D_0^{-1}C_2-B_2D_0^{-1}C_1$ of Theorem~\ref{thm:general-ber} is exactly the image of $B\otimes D_0^{-1}\otimes C$, contracted over the shared copy of $R_{\bar1}$ using $\omega$, and re-expressed in the $\theta_1\theta_2$-coefficient. In the scalar case $m=n=1$, $\Delta=\omega(\beta,\gamma)/(\theta_1\theta_2)\cdot d^{-1}$ reduces to the $2\times2$ determinant $p_1q_2-p_2q_1$ of Proposition~\ref{prop:11}, i.e.\ to $\omega$ applied directly to the coefficient vectors of $\beta,\gamma$. At general rank $r$ (Section~\ref{sec:rgen}), $\omega$ is replaced by the tautological alternating map $V\times V\to\Lambda^2(V)$, $(v,w)\mapsto v\wedge w$, which for $r=2$ is $\omega$ composed with the identification $\Lambda^2(V)\cong K\theta_1\theta_2$; the pairing $\Delta$ of Theorem~\ref{thm:rgen-ber} is the corresponding contraction, now valued in $\Lambda^2(V)$ rather than in $K$, and is genuinely not reducible to a scalar once $r\geq3$, since $\dim\Lambda^2(V)=\binom r2>1$.
\end{Remark}

\begin{Remark}[A Pl\"ucker-coordinate reading of $\Delta$]\label{rem:plucker}
The identification of $\Delta$ with an element of $\Lambda^2(V)$ (or, at general block size, of $M_m(K)\otimes\Lambda^2(V)$) invites comparison with the Pl\"ucker embedding of the Grassmannian $\mathrm{Gr}(2,V)$ of $2$-planes in $V$ into $\mathbb{P}(\Lambda^2V)$, where a $2$-plane $\mathrm{span}(b,c)\subset V$ is sent to the class of the decomposable bivector $b\wedge c$. We make this precise in the simplest sub-case $m=n=1$ at general rank $r$: an even $X\in M_{1|1}(R_r^{(2)})$ has scalar odd off-diagonal entries $\beta=\sum_ib^i\theta_i$, $\gamma=\sum_ic^i\theta_i$, i.e.\ $b=(b^1,\dots,b^r)$ and $c=(c^1,\dots,c^r)$ are literally vectors in $V$ under the basis identification $\theta_i\leftrightarrow e_i$. Theorem~\ref{thm:rgen-ber} then gives $\Delta=d_0^{-1}(b\wedge c)\in\Lambda^2(V)$ (with $d_0\in K^\times$ the body of $D$), so the coefficients $\Delta^{ij}=d_0^{-1}(b^ic^j-b^jc^i)$ are, up to the overall scalar $d_0^{-1}$, exactly the Pl\"ucker coordinates of the $2$-plane $\mathrm{span}(b,c)\subset V$ whenever $b,c$ are independent: $\Delta$ is a decomposable element of $\Lambda^2(V)$, and its class in $\mathbb{P}(\Lambda^2V)$ is precisely the image of $\mathrm{span}(b,c)$ under the Pl\"ucker embedding of $\mathrm{Gr}(2,V)\hookrightarrow\mathbb{P}(\Lambda^2V)$. At general block size $m,n$, $\Delta\in M_m(K)\otimes\Lambda^2(V)$ is the image of $B\otimes C\in\mathrm{Hom}(V,K^{m\times n})\otimes\mathrm{Hom}(V,K^{n\times m})$ under the map that composes matrices through $D_0^{-1}$ and then antisymmetrizes over $V\otimes V\to\Lambda^2(V)$; equivalently, it is the $M_m(K)$-family of Pl\"ucker-type data obtained by pairing, for each pair of row/column indices, the corresponding $1|1$ Pl\"ucker bivector as above. We do not develop the resulting super-Grassmannian geometry (e.g.\ an identification of $\Delta$ with a tangent vector to a super-Grassmannian at the point corresponding to $D$, in the spirit of the classical description of $\mathrm{Gr}(2,V)$'s tangent space via $\mathrm{Hom}$ of the tautological sub- and quotient bundles) beyond this coordinate-level observation, and leave it, as in the $1|1$ note, for future work.
\end{Remark}

\section{Related work: comparison with Khudaverdian--Voronov}\label{sec:relwork}

Khudaverdian and Voronov's theory of Berezinians via exterior powers and recurrent sequences~\cite{KV2005,KV2006} is the closest existing body of work to the present note, and a referee familiar with the Berezinian literature is likely to ask how the two relate. We summarize the comparison here to make the scope of our contribution (Remark~\ref{rem:scope}) precise.

\begin{itemize}\itemsep=0pt
\item \emph{Setting.} Khudaverdian--Voronov work with an arbitrary even linear operator $A$ on a general $p|q$-dimensional superspace $V$ over an unspecified (sufficiently general) ground ring, and study the characteristic function $R_A(z)=\mathrm{Ber}(1+zA)$ through its power expansions at $z=0$ and $z=\infty$. This note instead fixes a single, minimal ground ring, $R=K[\theta_1,\theta_2]$ (rank $r$ in Section~\ref{sec:rgen}), and studies an arbitrary even super-\emph{matrix} $X$ directly, not a one-parameter family $1+zA$.
\item \emph{Form of the answer.} Khudaverdian--Voronov show that the traces of the exterior powers $\Lambda^k A$ satisfy universal recurrence relations of period $q$, expressed via vanishing Hankel determinants in the Grothendieck ring of $\mathrm{GL}(V)$-representations, and use this to write $\mathrm{Ber}(A)$ as a \emph{ratio of polynomial invariants} of $A$ (traces of exterior powers), together with a superspace analogue of Cramer's rule. Our Theorems~\ref{thm:general-ber} and~\ref{thm:rgen-ber} instead give the $\theta$-expansion of $\mathrm{Ber}(X)$ directly in the entries of $X$'s blocks, as an explicit finite polynomial (linear in the top-degree part) with no auxiliary rational function or Hankel determinant construction; the two answers are of different logical type; one is coordinate- and representation-theoretic, the other is a closed-form coefficient computation for fixed, small nilpotence degree.
\item \emph{Multiplicativity.} Multiplicativity of the Berezinian is not the focus of~\cite{KV2005,KV2006}; it is used there as an input to the recurrence-relation machinery, not reproved from scratch. Theorem~\ref{thm:general-mult} (and its $r$-generator extension, Theorem~\ref{thm:rgen-mult}) instead gives a self-contained, block-level proof of multiplicativity for this specific family of ground rings, via the single trace-cyclicity identity $(\ast)$; this is a much narrower statement than the general theory but is proved independently of it.
\item \emph{Relation.} Because $R_r^{(2)}$ (or $R$, at $r=2$) is a specific, small ground ring, the abstract Khudaverdian--Voronov machinery applies to it as a special case and could in principle be specialized to re-derive our formulas; we have not carried out that specialization, and see the two approaches as complementary rather than competing: theirs gives structural, representation-theoretic control valid for any $p|q$ and any operator, ours gives a fully explicit, low-degree formula naturally suited to direct computation and to the second-order truncation of Section~\ref{sec:rgen}.
\end{itemize}

\section{Conclusion}

Passing from $K[\theta]/(\theta^2)$ to $K[\theta_1,\theta_2]$ removes the main degeneracy of the simplest exposition of super-linear algebra: the Berezinian's Schur-complement correction term $BD^{-1}C$, which vanishes identically with one odd generator, becomes a genuinely nonzero element controlled, at arbitrary block size, by the explicit matrix pairing $\Delta$ of Theorem~\ref{thm:general-ber}. We proved a closed-form Berezinian formula for arbitrary even $X\in M_{m|n}(R)$ (Theorem~\ref{thm:general-ber}) and a fully general multiplicativity theorem $\mathrm{Ber}(XY)=\mathrm{Ber}(X)\mathrm{Ber}(Y)$ (Theorem~\ref{thm:general-mult}), verified by an explicit trace-cyclicity cancellation rather than by appeal to the abstract theory, both specializing exactly to the previously known $1|1$ formulas (Corollaries~\ref{cor:recover11},~\ref{cor:recover11mult}) and illustrated on a fully worked $m=2,n=1$ numerical example (Example~\ref{ex:worked}). We also gave a fully general, sign-complete proof of graded cyclicity of the supertrace for equal-parity pairs and of $\mathrm{str}([A,B])=0$ for arbitrary block sizes $(m|n)$ (Theorems~\ref{thm:cyc},~\ref{thm:strzero}). We further extended the whole framework to an arbitrary number $r\geq2$ of odd generators (Section~\ref{sec:rgen}), with the scalar pairing $\Delta$ replaced by a $\Lambda^2(V)$-valued pairing (Theorems~\ref{thm:rgen-ber},~\ref{thm:rgen-mult}), identified in the scalar sub-case with a Pl\"ucker coordinate (Remark~\ref{rem:plucker}), and compared the resulting picture with the closest existing work, that of Khudaverdian and Voronov (Section~\ref{sec:relwork}). This resolves, in the direction of both block size and number of odd generators, the extension proposed at the end of the $1|1$ note. Natural next steps include removing the second-order truncation of Section~\ref{sec:rgen} to work over the full Grassmann algebra $\Lambda(K^r)$ for $r\geq4$, where $R_{\bar0}$ and $R_{\bar1}$ contain higher exterior powers and the formulas of Theorems~\ref{thm:rgen-ber}--\ref{thm:rgen-mult} would acquire further correction terms, and developing the super-Grassmannian interpretation of $\Delta$ sketched in Remark~\ref{rem:plucker}.

\subsection*{Acknowledgements}

The author thanks St Aloysius (Deemed to be University), Mangaluru, India, for support.

\pdfbookmark[1]{References}{ref}

\LastPageEnding


\begin{thebibliography}{99}
\footnotesize\itemsep=0pt

\bibitem{Berezin1987} Berezin F.A., \textit{Introduction to Superanalysis}, D. Reidel Publishing, 1987.

\bibitem{Leites1980} Leites D., Introduction to the theory of supermanifolds, \textit{Russian Math. Surveys} \textbf{35} (1980), no.~1, 1--64.

\bibitem{Manin1988} Manin Yu.I., \textit{Gauge Field Theory and Complex Geometry}, Grundlehren der mathematischen Wissenschaften 289, Springer, Berlin, 1988 (2nd ed.\ 1997).

\bibitem{Varadarajan2004} Varadarajan V.S., \textit{Supersymmetry for Mathematicians: An Introduction}, Courant Lecture Notes 11, American Mathematical Society, 2004.

\bibitem{DeligneMorgan1999} Deligne P., Morgan J.W., Notes on supersymmetry (following Joseph Bernstein), in \textit{Quantum Fields and Strings: A Course for Mathematicians}, Vol.~1, American Mathematical Society, 1999, 41--97.

\bibitem{Rogers2007} Rogers A., \textit{Supermanifolds: Theory and Applications}, World Scientific Publishing, Hackensack, NJ, 2007.

\bibitem{KV2005} Khudaverdian H.M., Voronov Th.Th., Berezinians, exterior powers and recurrent sequences, \textit{Lett. Math. Phys.} \textbf{74} (2005), no.~2, 201--228.

\bibitem{KV2006} Khudaverdian H.M., Voronov Th.Th., New facts about Berezinians, in \textit{Supersymmetries and Quantum Symmetries 2005}, Editors E. Ivanov, B. Zupnik, Dubna, 2006, 393--398.

\bibitem{Voronov1992} Voronov Th.Th., \textit{Geometric Integration Theory on Supermanifolds}, Harwood Academic Publishers, 1992 (2nd ed.\ Cambridge Scientific Publishers, 2014).

\bibitem{BBH1991} Bartocci C., Bruzzo U., Hern\'andez-Ruip\'erez D., \textit{The Geometry of Supermanifolds}, Mathematics and its Applications 71, Kluwer Academic Publishers, Dordrecht, 1991.

\end{thebibliography}
\end{document}